\documentclass{amsart}
\usepackage[latin1]{inputenc}
\usepackage{a4}
\usepackage{amsfonts,amssymb,amsmath,amsthm}
\usepackage{tikz}
\newcommand{\BD}{\operatorname{bd}}
\newcommand{\DIAM}{\operatorname{diam}}
\newcommand{\INFDIAM}{\operatorname{infdiam}}
\newcommand{\SKEL}{\operatorname{skel}}
\newcommand{\ZE}{\mathbb{Z}^2_{\rm even}}
\newtheorem{theorem}{Theorem}
\newtheorem{corollary}[theorem]{Corollary}
\newtheorem{lemma}[theorem]{Lemma}
\newtheorem{problem}[theorem]{Problem}
\newtheorem{proposition}[theorem]{Proposition}

\begin{document}


\title[Tilings of convex sets by incongruent equilateral triangles]{Tilings of convex sets by mutually incongruent equilateral triangles contain arbitrarily small tiles}

\author{Christian Richter}
\address{Institute of Mathematics, Friedrich Schiller University,
D-07737 Jena, Germany}
\email{christian.richter@uni-jena.de}
\author{Melchior Wirth}
\address{Institute of Mathematics, Friedrich Schiller University,
D-07737 Jena, Germany}
\email{melchior.wirth@uni-jena.de}

\date{\today}

\begin{abstract}
We show that every tiling of a convex set in the Euclidean plane $\mathbb{R}^2$ by equilateral triangles of mutually different sizes contains arbitrarily small tiles. The proof is purely elementary up to the discussion of one family of tilings of the full plane $\mathbb{R}^2$, which is based on a surprising connection to a random walk on a directed graph.
\end{abstract}

\maketitle


\section{Introduction and main result}

\emph{Can the Euclidean plane $\mathbb{R}^2$ be tiled by equilateral triangles of pairwise different sizes?} This problem, posed in \cite[Exercise~2.4.10]{gruenbaum1987} and \cite[Section~C11]{croft1991}, serves as the motivation for the present paper.

A \emph{tiling} of a set $A \subseteq \mathbb{R}^2$ is a family $\mathcal{T}=\{T_i:i \in I\}$ of subsets $T_i \subseteq A$ with mutually disjoint interiors such that $A = \bigcup \mathcal{T}= \bigcup_{i \in I} T_i$. The elements $T_i$ of $\mathcal{T}$ are called \emph{tiles}. We refer to \cite{gruenbaum1987} for a comprehensive survey on tilings in $\mathbb{R}^2$. Here we study tilings of convex sets $C$ by equilateral triangles $T_i$. The size of a triangle $T$ is measured by its \emph{diameter} $\DIAM(T)= \sup_{x,y \in T} \|x-y\|$, i.e., by its longest side length.
We are interested in the case when all tiles have pairwise different sizes. Such tilings by equilateral triangles are called \emph{perfect}. This name is adopted from perfect tilings by squares \cite{brooks1940}. 

Figure~\ref{fig_1} illustrates a family of tilings of the full plane $\mathbb{R}^2$ by equilateral triangles depending on a parameter $\alpha \in \left(0,\frac{1}{2}\right]$. These tilings are not perfect since there are only three (if $\alpha < \frac{1}{2}$) or two (if $\alpha = \frac{1}{2}$) different sizes of tiles. The limit case $\alpha \downarrow 0$ (with triangles of size $0$ being ignored) would give the hexagonal tiling by triangles of unit size.
\begin{figure}
\begin{center}
\begin{tikzpicture}[xscale=.4,yscale=.693]

\draw[line width=.5mm,densely dotted,->]
  (0,0) -- (5,1)
  ;

\draw[line width=.5mm,densely dotted,->]
  (0,0) -- (1,3)
  ;

\draw
  (-.5,-.5) -- (24.5,-.5) -- (24.5,11.5) -- (-.5,11.5) -- cycle

  (0,0) -- (6,0) -- (3,3) -- cycle 
  (5,1) -- (11,1) -- (8,4) -- cycle
  (10,2) -- (16,2) -- (13,5) -- cycle 
  (15,3) -- (21,3) -- (18,6) -- cycle
  (1,3) -- (7,3) -- (4,6) -- cycle 
  (6,4) -- (12,4) -- (9,7) -- cycle
  (11,5) -- (17,5) -- (14,8) -- cycle 
  (16,6) -- (22,6) -- (19,9) -- cycle
  (2,6) -- (8,6) -- (5,9) -- cycle 
  (7,7) -- (13,7) -- (10,10) -- cycle
  (12,8) -- (18,8) -- (15,11) -- cycle 
  (14,0) -- (20,0) -- (17,3) -- cycle

  (-.5,.5) -- (.5,-.5)
  (1.5,-.5) -- (2,0) -- (2.5,-.5)
  (5.5,-.5) -- (7,1) -- (8.5,-.5)
  (9.5,-.5) -- (12,2) -- (14.5,-.5)
  (15.5,-.5) -- (16,0) -- (16.5,-.5)
  (19.5,-.5) -- (21,1) -- (22.5,-.5) 
  (23.5,-.5) -- (24.5,.5) 
  (24.5,1) -- (19,1) -- (22,4) -- (24.5,1.5)
  (24.5,2) -- (24,2) -- (24.5,2.5)
  (24.5,4) -- (20,4) -- (23,7) -- (24.5,5.5)
  (24.5,7) -- (21,7) -- (24,10) -- (24.5,9.5)
  (24.5,10) -- (22,10) -- (23.5,11.5)
  (20.5,11.5) -- (23,9) -- (17,9) -- (19.5,11.5)
  (18.5,11.5) -- (19,11) -- (13,11) -- (13.5,11.5)
  (12.5,11.5) -- (14,10) -- (8,10) -- (9.5,11.5)
  (6.5,11.5) -- (9,9) -- (3,9) -- (5.5,11.5)
  (4.5,11.5) -- (5,11) -- (-.5,11)
  (-.5,9.5) -- (1,11) -- (4,8) -- (-.5,8)
  (-.5,7.5) -- (0,8) -- (3,5) -- (-.5,5)
  (-.5,4.5) -- (2,2) -- (-.5,2)

  (3,1) node {$1$}
  (1,-.33) node {$\alpha$}
  (8,2) node {$1$}
  (6,.67) node {$\alpha$}
  (9,.34) node {$1-\alpha$}
  (13,3) node {$1$}
  (11,1.67) node {$\alpha$}
  (14,1.34) node {$1-\alpha$}
  (18,4) node {$1$}
  (16,2.67) node {$\alpha$}
  (19,2.34) node {$1-\alpha$}

  (4,4) node {$1$}
  (2,2.67) node {$\alpha$}
  (5,2.34) node {$1-\alpha$}
  (9,5) node {$1$}
  (7,3.67) node {$\alpha$}
  (10,3.34) node {$1-\alpha$}
  (14,6) node {$1$}
  (12,4.67) node {$\alpha$}
  (15,4.34) node {$1-\alpha$}
  (19,7) node {$1$}
  (17,5.67) node {$\alpha$}
  (20,5.34) node {$1-\alpha$}

  (5,7) node {$1$}
  (3,5.67) node {$\alpha$}
  (6,5.34) node {$1-\alpha$}
  (10,8) node {$1$}
  (8,6.67) node {$\alpha$}
  (11,6.34) node {$1-\alpha$}
  (15,9) node {$1$}
  (13,7.67) node {$\alpha$}
  (16,7.34) node {$1-\alpha$}
  (20,10) node {$1$}
  (18,8.67) node {$\alpha$}
  (21,8.34) node {$1-\alpha$}

  (12,0) node {$1$} 
  (17,1) node {$1$} 
  (22,2) node {$1$} 
  (23,5) node {$1$} 
  (24,8) node {$1$} 
  (0,6) node {$1$} 
  (1,9) node {$1$} 
  (6,10) node {$1$} 
  (11,11) node {$1$} 

  (15,-.33) node {$\alpha$}
  (20,.67) node {$\alpha$}
  (21,3.67) node {$\alpha$}
  (22,6.67) node {$\alpha$}
  (23,9.67) node {$\alpha$}
  (4,8.67) node {$\alpha$}
  (9,9.67) node {$\alpha$}
  (14,10.67) node {$\alpha$}

  (23,.34) node {$1-\alpha$}
  (1,4.34) node {$1-\alpha$}
  (2,7.34) node {$1-\alpha$}
  (3,10.34) node {$1-\alpha$}
  (7,8.34) node {$1-\alpha$}
  (12,9.34) node {$1-\alpha$}
  (17,10.34) node {$1-\alpha$}
  ;
  
\end{tikzpicture}
\end{center}
\caption{A family of tilings of $\mathbb{R}^2$ with parameter $0 < \alpha \le \frac{1}{2}$. 
The numbers stand for sizes of tiles. Arrows indicate periodicity.}
\label{fig_1}
\end{figure}
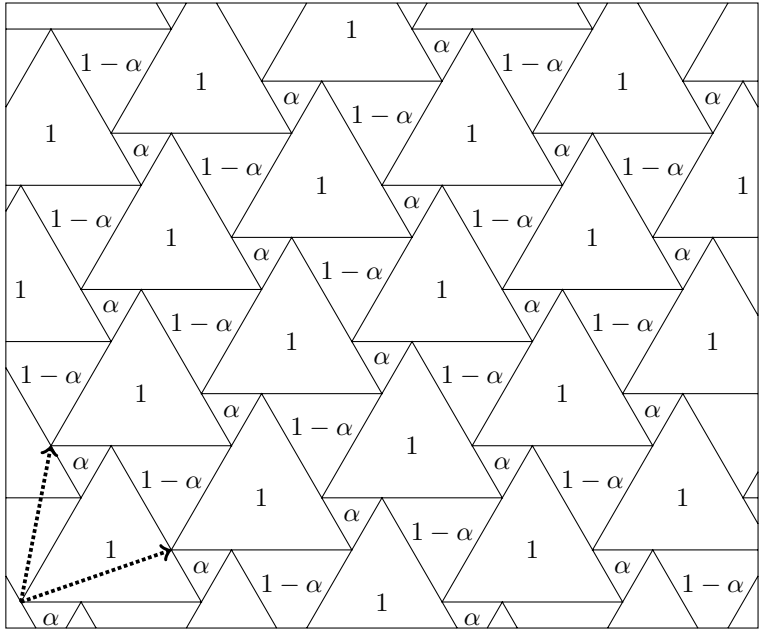

The property of perfectness is rather restrictive. Tutte \cite[p.\ 468]{tutte1948} showed that there is no perfect tiling $\mathcal{T}$ of an equilateral triangle $T$ by finitely many equilateral triangles apart from the trivial one $\mathcal{T}=\{T\}$. This was generalized to tilings of bounded convex sets by Buchman \cite{buchman1981} and by Tuza \cite{tuza1991} (see also \cite[Theorem~6]{kupavskii2017}). Scherer \cite{scherer1983} complemented this by a negative result on tilings 
of $\mathbb{R}^2$ and motivated the above introductory question.
\begin{theorem}\label{thm_old}
{\bf (a) } There is no convex subset of $\mathbb{R}^2$ that possesses a perfect tiling by finitely many and at least two equilateral triangles \cite{buchman1981, tuza1991}.

{\bf (b) } A perfect tiling of the plane $\mathbb{R}^2$ by equilateral triangles does not contain a smallest tile \cite{scherer1983}.
\end{theorem}

We shall show the following result, thereby answering a question by Nandakumar \cite[4.~Note~2]{nandacumar2016}.
\begin{theorem}\label{thm_main}
If a perfect tiling of a convex subset of the plane by equilateral triangles consists of at least two tiles, then it contains arbitrarily small tiles.
\end{theorem}
Theorem~\ref{thm_main} implies Theorem~\ref{thm_old}. The generalization of Theorem~\ref{thm_old}(b) is strict in so far as it excludes tilings of $\mathbb{R}^2$ such that the infimum of the sizes of all tiles is positive, but not attained as a minimum. The proof will be given in the following section. In parts it adopts an idea of Scherer \cite{scherer1983}, but does not make use of Theorem~\ref{thm_old}. This way it gives an alternative approach to Theorem~\ref{thm_old}(a) (see Subsections~\ref{subsec_2} and \ref{subsec_3} below).

Often one requires a tiling to be \emph{locally finite}, that is, every bounded subset of $\mathbb{R}^2$ meets at most finitely many tiles. For us local finiteness is not an a priory assumption. The general setting is much more flexible. In particular, it allows for perfect tilings by equilateral triangles of the plane $\mathbb{R}^2$ and of every open subset of $\mathbb{R}^2$ \cite[Corollary~2]{richter2012}. A convex polygon can be perfectly tiled if and only if all its inner angles are not smaller than $\frac{\pi}{3}$
\cite[Corollary~7]{richter201?}. In fact, if a subset of $\mathbb{R}^2$ possesses a tiling by equilateral triangles, then it possesses a perfect tiling by equilateral triangles \cite[Theorem~1]{richter201?}.

Since all the tilings mentioned in the last paragraph exhibit accumulation phenomena, we formulate the remaining open question explicitly.
\begin{problem}\label{prob}
Does the plane $\mathbb{R}^2$ itself or any of its unbounded convex subsets possess a locally finite perfect tiling by equilateral triangles?
\end{problem}

As byproducts of the proof of Theorem~\ref{thm_main}, we obtain a further necessary topological condition that is satisfied by every locally finite perfect tiling of an unbounded convex subset of $\mathbb{R}^2$ by equilateral triangles (see Corollary~\ref{cor_condition}), and we characterize all locally finite tilings of (not necessarily unbounded) convex sets by (not necessarily incongruent) equilateral triangles that do not meet that condition (see Propositions~\ref{prop_1} and \ref{prop_2}).

Moreover, we shall prove a result in the spirit of two theorems from \cite{kupavskii2017}, which say that every locally finite tiling of the plane by arbitrary triangles with side lengths in the interval $[1,2)$ contains two triangles that share a side \cite[Theorem~5]{kupavskii2017} and that every finite tiling of a convex $k$-gon, $k \ge 4$, with arbitrary triangles contains two triangles that share a side \cite[Theorem~6]{kupavskii2017}.

\begin{theorem}\label{thm_shared-side}
Let $\mathcal{T}$ be a tiling of a convex subset of the plane by at least two equilateral triangles with a positive lower bound on the side lengths of all tiles. Then $\mathcal{T}$ contains two triangles that share a side or $\mathcal{T}$ is similar to a tiling of the type depicted in Figure~\ref{fig_1} with suitable $\alpha \in \left(0,\frac{1}{2}\right]$.
\end{theorem}

Theorem~\ref{thm_shared-side} improves on a recent result on tilings of the complete plane $\mathbb{R}^2$ by Aduddell, Ascanio, Deaton and Mann \cite[Theorem~1]{aduddell2017} in two ways. Firstly they assume that the tiling contains only tiles of finitely many different sizes, whereas we only impose a positive lower bound on the size of the tiles. Secondly their result requires an additional symmetry condition called equitransitivity, which we do not need at all.


\section{Proofs}


\subsection{Local finiteness}

\begin{lemma}\label{lem_1}
If a tiling $\mathcal{T}$ by equilateral triangles is not locally finite, then it contains arbitrarily small tiles.
\end{lemma}

\begin{proof}
We use the notation $D(x_0,\varrho)=\{x \in \mathbb{R}^2: \|x-x_0\| \le \varrho\}$ for the Euclidean disc with center $x_0 \in \mathbb{R}^2$ and radius $\varrho > 0$. Let $\INFDIAM(\mathcal{T})=\inf_{T \in \mathcal{T}} \DIAM(T)$.

Since $\mathcal{T}$ is not locally finite, there exist a disc $D(x_0,\varrho)$ and a sequence of distinct triangles $T_1,T_2,\ldots \in \mathcal{T}$ such that we can pick $x_i \in D(x_0,\varrho) \cap T_i$, $i=1,2,\ldots$ Let $T'_i \subseteq T_i$ be obtained from $T_i$ by a dilation with center $x_i$ and factor $\frac{\INFDIAM(\mathcal{T})}{\DIAM(T_i)} \le 1$. (Here degeneracy with factor $0$ is not forbidden.) Then $\DIAM(T'_i)=\INFDIAM(\mathcal{T})$ and $T'_i \subseteq D(x_0,\varrho+\INFDIAM(\mathcal{T}))$ by the triangle inequality, because $x_i \in D(x_0,\varrho) \cap T'_i$. Consequently, $D(x_0,\varrho+\INFDIAM(\mathcal{T}))$ contains infinitely many non-overlapping equilateral triangles $T'_i$ of size $\INFDIAM(\mathcal{T})$. A volumetric argument yields our claim $\INFDIAM(\mathcal{T})=0$.
\end{proof}

Lemma~\ref{lem_1} on tilings by similar images of an equilateral triangle $T$ sounds trivial, but Figure~\ref{fig_1a} illustrates that tilings by \emph{affine} images of $T$ may fail to be locally finite without containing arbitrarily small pieces. 

\begin{figure}
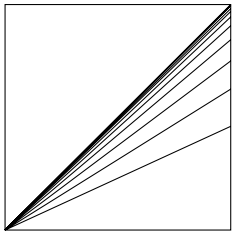


\tikz[xscale=3,yscale=3]{
\draw
  (0,0) -- (1,0) -- (1,1) -- (0,1) -- cycle
  (0,0) -- (1,1)
  (0,0) -- (1,.46)
  (0,0) -- (1,.625)
  (0,0) -- (1,.75)
  (0,0) -- (1,.844)
  (0,0) -- (1,.906)
  (0,0) -- (1,.945)
  (0,0) -- (1,.969)
  (0,0) -- (1,.982)
  (0,0) -- (1,.995)
  (0,0) -- (1,.997)
  (0,0) -- (1,.998)
  (0,0) -- (1,.999)
;}
\caption{Tiling of a square by infinitely many triangles.\label{fig_1a}}

\end{figure}


\subsection{$E$-configurations\label{subsec_2}}

Now we can restrict our consideration to locally finite tilings $\mathcal{T}$ of convex sets.
Let $\BD(T)$ denote the boundary of a triangle $T$. The \emph{skeleton} $\SKEL(\mathcal{T})= \bigcup_{T \in \mathcal{T}} \BD(T)$ of a tiling $\mathcal{T}$ by equilateral triangles is a union of segments of only three directions; say $v_E=(1,0)$ (east), $v_{NE}=\left(\frac{1}{2},\frac{\sqrt{3}}{2}\right)$ (north east) and $v_{NW}=\left(-\frac{1}{2},\frac{\sqrt{3}}{2}\right)$ (north west). 

Next we follow an idea of Scherer \cite{scherer1983}. We denote the closed line segment between $x,y \in \mathbb{R}^2$ by $[x,y]$. For every $\varepsilon > 0$ and every $0 < \mu < 1$, we call an isometric image of the set
\begin{equation}\label{eq_1}
[(0,0),\lambda v_E] \cup [(0,0), \varepsilon v_{NW}] \cup [\mu\lambda v_E,\mu\lambda v_E+ \varepsilon v_{NW}] \cup [\lambda v_E,\lambda v_E+ \varepsilon v_{NW}]
\end{equation}
an \emph{$E$-configuration of length $\lambda > 0$} (see the left-hand part of Figure~\ref{fig_2}). The respective isometric image of the segment $[(0,0),\lambda v_E]$ will be called the \emph{basis} of the $E$-configuration.
We say that $\mathcal{T}$ possesses an $E$-configuration of length $\lambda$ if $\SKEL(\mathcal{T})$ contains such a configuration as a subset. 
\begin{figure}
\begin{center}
\begin{tikzpicture}[xscale=.26,yscale=.45]

\draw
  (-1,1) -- (0,0) -- (16,0) -- (15,1)
  (8,0) -- (7,1)
  
  (0,-.6) node {\small $(0,0)$}
  (-1,1) node[above] {\small $\varepsilon v_{NW}$}
  (8,-.6) node {\small $\mu\lambda v_E$}
  (7,1) node[above] {\small $\mu\lambda v_E+\varepsilon v_{NW}$}
  (16,-.6) node {\small $\lambda v_E$}
  (15,1) node[above] {\small $\lambda v_E+\varepsilon v_{NW}$}
  
  (22,1) -- (23,0) -- (39,0) -- (38,1)
  (23,0) -- (27,4) -- (31,0) -- (33,2) -- (35,0)
  
  (27,1.33) node {$T_1$}
  (33,.66) node {$T_2$}
  (23,-.6) node {$x_1=(0,0)$}
  (31,-.7) node {$x_2$}
  (25.6,4) node {$x_3$}
  (35,-.7) node {$x_4$}
  (33,2) node[above] {$x_5$}
  (39,-.6) node {$\lambda v_E$}
  ;
  
\draw[densely dotted]
  (27,4) -- (29,4) node[right] {Case 1.}
  (27,4) -- (26,5) node[above] {Case 2.}
  ;
  
\end{tikzpicture}
\end{center}
\caption{An $E$-configuration and the proof of Lemma~\ref{lem_2}.}
\label{fig_2}
\end{figure}
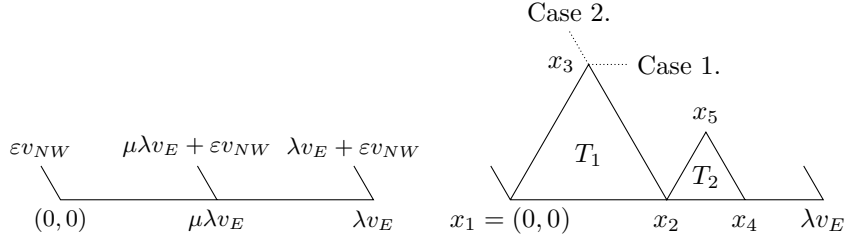

\begin{lemma}[cf. Scherer \cite{scherer1983}]\label{lem_2}
If a locally finite tiling $\mathcal{T}$ of a convex set $C \subseteq \mathbb{R}^2$ by equilateral triangles possesses an $E$-configuration and does not contain two triangles sharing a side, then $\mathcal{T}$ contains arbitrarily small tiles.
\end{lemma}

\begin{proof}
W.l.o.g. the tiling $\mathcal{T}$ possesses the $E$-configuration $E_0$ from (\ref{eq_1}). Since $\mathcal{T}$ covers the convex hull of $E_0$, there are triangles $T_1=\triangle x_1x_2x_3, T_2=\triangle x_2x_4x_5 \in \mathcal{T}$ such that $x_1=(0,0)$, $x_2 \in [(0,0),\lambda v_E]$, $x_4 \in [x_2,\lambda v_E]$ and $x_3$ as well as $x_5$ have positive second coordinates (see the right-hand part of Figure~\ref{fig_2}). 

If $\SKEL(\mathcal{T})$ contains a segment starting from $x_3$ in direction $v_E$ (Case 1 in Figure~\ref{fig_2}), then there is another $E$-configuration $E_1$ with basis $[x_2,x_3]$ since $[x_2,x_2]$ cannot be the side of another tile from $\mathcal{T}$. If not, then $\SKEL(\mathcal{T})$ must contain a segment starting from $x_3$ in direction $v_{NW}$ (Case 2 in Figure~\ref{fig_2}) and we obtain an $E$-configuration $E_1$ with basis $[x_1,x_3]$ by the same argument.
In either case the length of $E_1$ is $\lambda_1=\DIAM(T_1)$. Since the length $\lambda$ of $E_0$ satisfies
$$
\lambda \ge \DIAM(T_1)+\DIAM(T_2) \ge \lambda_1+\INFDIAM(\mathcal{T}),
$$
the $E$-configuration $E_1$ has a length $\lambda_1 \le \lambda - \INFDIAM(\mathcal{T})$.

Iterating this procedure, we obtain $E$-configurations $E_k$, $k=1,2,\ldots$, of lengths $\lambda_k \le \lambda- k \INFDIAM(\mathcal{T})$. Since all these lengths are positive, the claim $\INFDIAM(\mathcal{T})=0$ follows. 
\end{proof}


\subsection{Tilings of bounded convex sets without $E$-configurations\label{subsec_3}}

\begin{proposition}\label{prop_1}
Let $\mathcal{T}$ be a finite tiling of a bounded convex set $C \subseteq \mathbb{R}^2$ by at least two equilateral triangles and assume that $\mathcal{T}$ does not contain an $E$-configuration. Then $\mathcal{T}$ is the image of one of the five tilings of polygons from Figure~\ref{fig_3} under a similarity transformation.
\end{proposition}
\begin{figure}
\begin{center}
\begin{tikzpicture}[xscale=.37,yscale=.641]

\draw 
  (0,0) -- (4,0) -- (2,2) -- cycle
  (2,0) -- (3,1) -- (1,1) -- cycle
  
  (9,0) -- (11,2) -- (7,2) -- (5,0) -- (9,0) -- (7,2)
  
  (14,2) -- (12,0) -- (20,0) -- (18,2) -- (14,2) -- (16,0) -- (18,2)
  
  (23,2) -- (21,0) -- (25,0) -- (26,1) -- (25,2) -- (23,2) -- (25,0)
  (25,2) -- (24,1) -- (26,1)
  
  (27,1) -- (28,0) -- (30,0) -- (31,1) -- (30,2) -- (28,2) -- (27,1) -- (31,1)
  (28,0) -- (30,2)
  (28,2) -- (30,0)
  ;
  
\end{tikzpicture}
\end{center}
\caption{Finite tilings of bounded convex sets without $E$-configurations.}
\label{fig_3}
\end{figure}
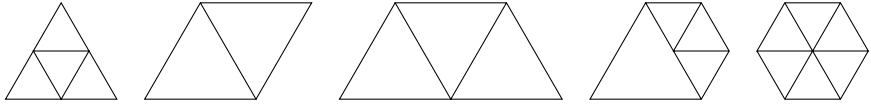

\begin{proof}
Here $C$ is a convex polygon with inner angles of sizes $\frac{\pi}{3}$ or $\frac{2\pi}{3}$. Hence $C$ can be a triangle (three inner angles of $\frac{\pi}{3}$), a quadrilateral (two angles of $\frac{\pi}{3}$ and two of $\frac{2\pi}{3}$) with parallelogram and trapezoid as subcases, a pentagon (one angle of $\frac{\pi}{3}$ and four of $\frac{2\pi}{3}$) or a hexagon (six angles of $\frac{2\pi}{3}$).

Since $\mathcal{T}$ does not contain an $E$-configuration, the sides of $C$ satisfy the following.
\begin{itemize}
\item[(i)] If at least one endpoint of a side $S$ of $C$ corresponds to an inner angle of $\frac{2\pi}{3}$, then $S$ is a side of only one triangle of $\mathcal{T}$.
\item[(ii)] If a side $S$ of $C$ connects two vertices with inner angles of $\frac{\pi}{3}$, then there are at most two triangles in $\mathcal{T}$ each having one side in $S$.
\end{itemize}

Discussing the five cases for $C$ mentioned above under the restrictions (i) and (ii) gives the five tilings in Figure~\ref{fig_3}. (The trivial tiling $\mathcal{T}=\{C\}$ for the case when $C$ is a triangle is excluded by the assumption of Proposition~\ref{prop_1}).
\end{proof}

Note that Lemma~\ref{lem_2} and Proposition~\ref{prop_1} give a new elementary approach to Theorem~\ref{thm_old}(a).


\subsection{Tilings of unbounded convex sets without $E$-configurations}

\begin{proposition}\label{prop_2}
Let $\mathcal{T}$ be a locally finite tiling of an unbounded convex set $C \subseteq \mathbb{R}^2$ by equilateral triangles and suppose that $\mathcal{T}$ does not contain an $E$-configuration. 
Then $C=\mathbb{R}^2$ and $\mathcal{T}$ is the image of one of the tilings illustrated in Figure~\ref{fig_1} under a similarity transformation.
\end{proposition}

The proof of Proposition~\ref{prop_2} is based on two lemmas. The first one is elementary and concerns the topology of $\mathcal{T}$. The second one is given in the following subsection. We call two tilings $\mathcal{T}_1$ and $\mathcal{T}_2$ \emph{topologically equivalent} (or \emph{having the same topology}) if there is a bicontinuous bijection $\Phi: \bigcup \mathcal{T}_1 \rightarrow \bigcup \mathcal{T}_2$ such that $\mathcal{T}_2=\{\Phi(T): T \in \mathcal{T}_1\}$ (cf. \cite[p.~167]{gruenbaum1987}).

\begin{lemma}\label{lem_3}
Let $\mathcal{T}$ be a locally finite tiling of an unbounded convex set $C \subseteq \mathbb{R}^2$ by equilateral triangles and suppose that $\mathcal{T}$ does not contain an $E$-configuration. 
Then $C=\mathbb{R}^2$ and $\mathcal{T}$ contains a tiling $\mathcal{T}_0 \subseteq \mathcal{T}$ that has the same topology as the one from Figure~\ref{fig_1}.
\end{lemma}

\begin{proof}
\emph{(A). Determination of $C$. } Suppose that $C \ne \mathbb{R}^2$. Since $\mathcal{T}$ is locally finite and $C$ is unbounded, the boundary of $C$ contains a line or a half-line. There are infinitely many $E$-configurations having their bases in that line or half-line. So this case cannot appear, and we have $C=\mathbb{R}^2$.

\emph{(B). Topology of $\mathcal{T}$ in the neighbourhood of a maximal segment from $\SKEL(\mathcal{T})$. } Let $S$ be a line segment in $\SKEL(\mathcal{T})$ that is maximal under inclusion. By the argument of (A), $S$ is bounded.
W.l.o.g., $S=[x_W,x_E]$ is parallel to $v_E$ (see Subsection~\ref{subsec_2}) and $x_W$ and $x_E$ are the west and east endpoints of $S$, respectively. Because of maximality of $S$, the endpoints $x_W,x_E$ are relatively inner points of two maximal segments $S_W, S_E \subseteq \SKEL(\mathcal{T})$, and they are in the relative interiors of sides of triangles $T_W,T_E \in \mathcal{T}$, respectively (see the left-hand part of Figure~\ref{fig_4}; a grey half-disc indicates that this region belongs to one single tile). 
\begin{figure}
\begin{center}
\begin{tikzpicture}[xscale=.3,yscale=.52]

\draw
  (0,0) node {\begin{tikzpicture}
    \fill [fill=white]
      (0mm,0mm) circle (3mm)
      ;
    \fill [fill=black!30]
      (-1.25mm,2.165mm) arc (120:300:2.5mm)
      ;
  \end{tikzpicture}}
  ;

\draw
  (12,0) node {\begin{tikzpicture}
    \fill [fill=white]
      (0mm,0mm) circle (3mm)
      ;
    \fill [fill=black!30]
      (1.25mm,-2.165mm) arc (-60:120:2.5mm)
      ;
  \end{tikzpicture}}
  ;

\draw
  (0,0) -- (12,0) 
  (-.5,.5) --  (2,-2) -- (4,0) 
  (0,0) -- (2.5,2.5) -- (5,0)
  (7,0) -- (9.5,2.5) -- (12.5,-.5)
  (6,0) -- (9,-3) -- (12,0)
  (-.6,.6) -- (-.7,.7) 
  (-.8,.8) -- (-.9,.9)
  (2.1,-2.1) -- (2.2,-2.2) 
  (2.3,-2.3) -- (2.4,-2.4)
  (9.4,2.6) -- (9.3,2.7) 
  (9.2,2.8) -- (9.1,2.9)
  (12.6,-.6) -- (12.7,-.7) 
  (12.8,-.8) -- (12.9,-.9)

  (6,0) node[above] {$S$}
  (6,1.2) node {$\cdots$}
  (5,-1.2) node {$\cdots$}
 
  (-1.1,.9) node[above] {$S_W$}
  (-1.2,-.6) node {$T_W$}
  (13.1,-.9) node[below] {$S_E$}
  (13.5,.5) node {$T_E$}
  
  (2.5,.6) node {$T_{N,1}$}
  (9.5,.6) node {$T_{N,n_N}$}
  (2,-.6) node {$T_{S,1}$}
  (9,-.6) node {$T_{S,n_S}$}
  ;
  

\draw
  (25,1.46) node {\begin{tikzpicture}[xscale=.3,yscale=.52]

\draw
  (6,0) node {\begin{tikzpicture}
    \fill [fill=white]
      (0mm,0mm) circle (3mm)
      ;
    \fill [fill=black!30]
      (-1.25mm,2.165mm) arc (120:-60:2.5mm)
      ;
  \end{tikzpicture}}
  ;

\draw
  (3,3) node {\begin{tikzpicture}
    \fill [fill=white]
      (0mm,0mm) circle (3mm)
      ;
    \fill [fill=black!30]
      (1.25mm,2.165mm) arc (60:240:2.5mm)
      ;
  \end{tikzpicture}}
  ;

\draw
  (8,-2) node {\begin{tikzpicture}
    \fill [fill=white]
      (0mm,0mm) circle (3mm)
      ;
    \fill [fill=black!30]
      (-2.5mm,0mm) arc (-180:0:2.5mm)
      ;
  \end{tikzpicture}}
  ;

\draw
  (0,0) -- (6,0) -- (3,-3) -- (0,0) -- (5,5) -- (7,3)
  (3,3) -- (13,3) -- (8,-2) -- cycle
  (6.5,-2) -- (9.5,-2)

  (0,0) node[left] {$x_W$}
  (5,5) node[above] {$x^\ast$}
  (3,0) node[above] {$S$}
  (6.2,1) node {$S_E$}
  (3,1.5) node {$T_{N,1}$}
  (3,-1) node {$T_{S,1}$}
  (8.5,1.5) node {$T_E$}
  (5,3.66) node {$T^\ast$}  
  ;

  \end{tikzpicture}}
  ;

\end{tikzpicture}
\end{center}
\caption{Notations in the neighbourhood of a maximal segment $S \subseteq \SKEL(\mathcal{T})$ and
contradiction in Case 2.}
\label{fig_4}
\end{figure}
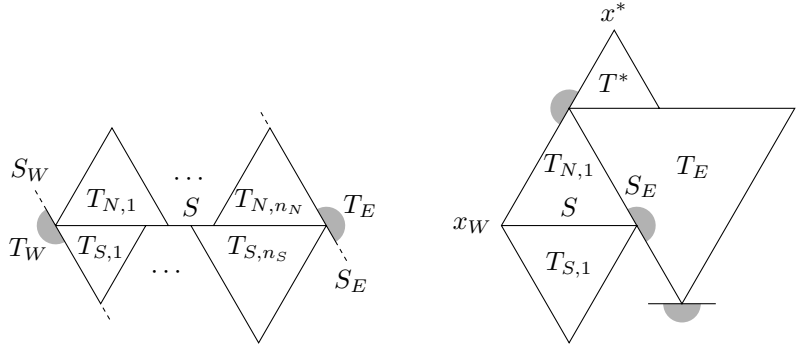

There are uniquely determined tiles $T_{N,1},\ldots,T_{N,n_N} \in \mathcal{T}$ and $T_{S,1},\ldots,T_{S,n_S} \in \mathcal{T}$ such that $S$ is the union of the south sides of $T_{N,1},\ldots,T_{N,n_N}$ as well as of the north sides of $T_{S,1},\ldots,T_{S,n_S}$. We can assume that 
$$
n_N \le n_S.
$$

\emph{Case 1: $n_S \ge 2$. } First note that $n_S \le 2$, since otherwise there would be an $E$-configuration with basis $S \cap \left(T_{S,1} \cup T_{S,2}\right)$. So we have $n_S=2$. Next note that $S_W$ must be parallel to $v_{NW}$ and $S_E$ parallel to $v_{NE}$, because otherwise we would again have an $E$-configuration with basis $S \cap \left(T_{S,1} \cup T_{S,2}\right)$. Then $n_N=1$, since $n_N \ge 2$ would give an $E$-configuration with basis $S \cap \left(T_{N,1} \cup T_{N,2}\right)$. So $n_N=1$, $n_S=2$, and the resulting topology is displayed in Figure~\ref{fig_5}.

\emph{Case 2: $n_S=1$. } Then $n_N=n_S=1$ and $S$ is a joint edge of $T_{N,1}$ and $T_{S,1}$. W.l.o.g., $S_E$ contains a side of $T_{N,1}$ (see the right-hand part of Figure~\ref{fig_4}). 
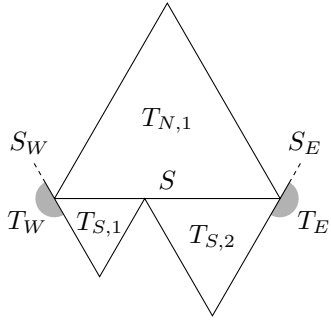
\begin{figure}
\begin{center}
\begin{tikzpicture}[xscale=.3,yscale=.52]

\draw
  (20,0) node {\begin{tikzpicture}
    \fill [fill=white]
      (0mm,0mm) circle (3mm)
      ;
    \fill [fill=black!30]
      (-1.25mm,2.165mm) arc (120:300:2.5mm)
      ;
  \end{tikzpicture}}
  ;

\draw
  (30,0) node {\begin{tikzpicture}
    \fill [fill=white]
      (0mm,0mm) circle (3mm)
      ;
    \fill [fill=black!30]
      (-1.25mm,-2.165mm) arc (-120:60:2.5mm)
      ;
  \end{tikzpicture}}
  ;

\draw
  (20,0) -- (30,0)
  (20,0) -- (25,5) -- (30,0)
  (19.5,.5) -- (22,-2) -- (24,0) -- (27,-3) -- (30.5,.5)
  (19.4,.6) -- (19.3,.7)
  (19.2,.8) -- (19.1,.9) 
  (30.6,.6) -- (30.7,.7)
  (30.8,.8) -- (30.9,.9) 

  (25,0) node[above] {$S$}
 
  (18.9,.9) node[above] {$S_W$}
  (18.8,-.6) node {$T_W$}
  (31.1,.9) node[above] {$S_E$}
  (31.5,-.6) node {$T_E$}
  
  (25,2) node {$T_{N,1}$}
  (22,-.66) node {$T_{S,1}$}
  (27,-1) node {$T_{S,2}$}
  ;

\end{tikzpicture}
\end{center}
\caption{Unique topology close to a maximal segment $S \subseteq \SKEL(\mathcal{T})$.}
\label{fig_5}
\end{figure}
The maximal segment $S_E$ contains sides of at least two triangles of $\mathcal{T}$, both being on the same side of $S_E$ as $S$. Hence $S_E$ is in the situation of Case 1 and the topology in its neighbourhood is as in Figure~\ref{fig_5}. In particular, $S_E$ is a side of $T_E$ and $S_E$ is bounded at its north west endpoint by a segment in direction $v_{NE}$ and at its south east endpoint by a segment in direction $v_{E}$. Let $T^\ast \in \mathcal{T}$ be the tile whose south west vertex coincides with the north vertex of $T_{N,1}$, and let $x^\ast$ be the north vertex of $T^\ast$. We obtain an $E$-configuration with basis $[x_W,x^\ast]$. This contradiction shows that Case 2 does not appear.

We have shown that \emph{the topology in a neighbourhood of every maximal line segment in $\SKEL(\mathcal{T})$ is as illustrated in Figure~\ref{fig_5}.}

\emph{(C). Construction of $\mathcal{T}_0$. } We shall identify triangles $T_{i,j},L_{i,j},R_{i,j} \in \mathcal{T}$ for $(i,j) \in \ZE=\{(i,j) \in \mathbb{Z}^2: i+j \text{ is even}\}$. The triangles $T_{i,j}$ are the equivalents of the triangles of size $1$ from Figure~\ref{fig_1}, so that $T_{i,j+2}$, $T_{i-1,j-1}$ and $T_{i+1,j-1}$ meet the north, south west and south east vertex of $T_{i,j}$. The triangles $L_{i,j}$ and $R_{i,j}$ are the equivalents of the triangles of sizes $\alpha$ and $1-\alpha$, respectively, which touch $T_{i,j}$ along their north sides.
The construction is illustrated in Figure~\ref{fig_6}. Pieces of information obtained in Steps 0, 1 and 2 are marked by \tikz{ \draw (0,0) circle (1.2mm) (0,-0) node {\scriptsize 0};}, \tikz{ \draw (0,0) circle (1.2mm) (0,0) node {\scriptsize 1};} and \tikz{ \draw (0,0) circle (1.2mm) (0,0) node {\scriptsize 2};}, respectively.
\begin{figure}
\begin{center}
\begin{tikzpicture}[xscale=.6,yscale=1.04]

\draw
  (0,2) node {\begin{tikzpicture}
    \fill [fill=white]
      (0mm,0mm) circle (3.5mm)
      ;
    \fill [fill=black!30]
      (3mm,0mm) arc (0:180:3mm)
      ;
  \end{tikzpicture}}
  ;
\draw
  (-3,1) node {\begin{tikzpicture}
    \fill [fill=white]
      (0mm,0mm) circle (3.5mm)
      ;
    \fill [fill=black!30]
      (3mm,0mm) arc (0:180:3mm)
      ;
  \end{tikzpicture}}
  ;
\draw
  (3,1) node {\begin{tikzpicture}
    \fill [fill=white]
      (0mm,0mm) circle (3.5mm)
      ;
    \fill [fill=black!30]
      (3mm,0mm) arc (0:180:3mm)
      ;
  \end{tikzpicture}}
  ;
\draw
  (-6,0) node {\begin{tikzpicture}
    \fill [fill=white]
      (0mm,0mm) circle (3.5mm)
      ;
    \fill [fill=black!30]
      (2.5mm,0mm) arc (0:180:2.5mm)
      ;
  \end{tikzpicture}}
  ;
\draw
  (-3,3) node {\begin{tikzpicture}
    \fill [fill=white]
      (0mm,0mm) circle (3.5mm)
      ;
    \fill [fill=black!30]
      (2.5mm,0mm) arc (0:180:2.5mm)
      ;
  \end{tikzpicture}}
  ;
\draw
  (0,4) node {\begin{tikzpicture}
    \fill [fill=white]
      (0mm,0mm) circle (3.5mm)
      ;
    \fill [fill=black!30]
      (2.5mm,0mm) arc (0:180:2.5mm)
      ;
  \end{tikzpicture}}
  ;
\draw
  (3,3) node {\begin{tikzpicture}
    \fill [fill=white]
      (0mm,0mm) circle (3.5mm)
      ;
    \fill [fill=black!30]
      (2.5mm,0mm) arc (0:180:2.5mm)
      ;
  \end{tikzpicture}}
  ;
\draw
  (6,0) node {\begin{tikzpicture}
    \fill [fill=white]
      (0mm,0mm) circle (3.5mm)
      ;
    \fill [fill=black!30]
      (2.5mm,0mm) arc (0:180:2.5mm)
      ;
  \end{tikzpicture}}
  ;

\draw
  (2,0) node {\begin{tikzpicture}
    \fill [fill=white]
      (0mm,0mm) circle (3.5mm)
      ;
    \fill [fill=black!30]
      (1.5mm,2.6mm) arc (60:-120:3mm)
      ;
  \end{tikzpicture}}
  ;
\draw
  (-1,-1) node {\begin{tikzpicture}
    \fill [fill=white]
      (0mm,0mm) circle (3.5mm)
      ;
    \fill [fill=black!30]
      (1.5mm,2.6mm) arc (60:-120:3mm)
      ;
  \end{tikzpicture}}
  ;
\draw
  (2,2) node {\begin{tikzpicture}
    \fill [fill=white]
      (0mm,0mm) circle (3.5mm)
      ;
    \fill [fill=black!30]
      (1.5mm,2.6mm) arc (60:-120:3mm)
      ;
  \end{tikzpicture}}
  ;
\draw
  (-4,-2) node {\begin{tikzpicture}
    \fill [fill=white]
      (0mm,0mm) circle (3.5mm)
      ;
    \fill [fill=black!30]
      (1.25mm,2.165mm) arc (60:-120:2.5mm)
      ;
  \end{tikzpicture}}
  ;
\draw
  (2,-2) node {\begin{tikzpicture}
    \fill [fill=white]
      (0mm,0mm) circle (3.5mm)
      ;
    \fill [fill=black!30]
      (1.25mm,2.165mm) arc (60:-120:2.5mm)
      ;
  \end{tikzpicture}}
  ;
\draw
  (5,-1) node {\begin{tikzpicture}
    \fill [fill=white]
      (0mm,0mm) circle (3.5mm)
      ;
    \fill [fill=black!30]
      (1.25mm,2.165mm)arc (60:-120:2.5mm)
      ;
  \end{tikzpicture}}
  ;
\draw
  (5,1) node {\begin{tikzpicture}
    \fill [fill=white]
      (0mm,0mm) circle (3.5mm)
      ;
    \fill [fill=black!30]
      (1.25mm,2.165mm) arc (60:-120:2.5mm)
      ;
  \end{tikzpicture}}
  ;
\draw
  (2,4) node {\begin{tikzpicture}
    \fill [fill=white]
      (0mm,0mm) circle (3.5mm)
      ;
    \fill [fill=black!30]
      (1.25mm,2.165mm) arc (60:-120:2.5mm)
      ;
  \end{tikzpicture}}
  ;

\draw
  (-2,0) node {\begin{tikzpicture}
    \fill [fill=white]
      (0mm,0mm) circle (3.5mm)
      ;
    \fill [fill=black!30]
      (-1.5mm,2.6mm) arc (120:300:3mm)
      ;
  \end{tikzpicture}}
  ;
\draw
  (-2,2) node {\begin{tikzpicture}
    \fill [fill=white]
      (0mm,0mm) circle (3.5mm)
      ;
    \fill [fill=black!30]
      (-1.5mm,2.6mm) arc (120:300:3mm)
      ;
  \end{tikzpicture}}
  ;
\draw
  (1,-1) node {\begin{tikzpicture}
    \fill [fill=white]
      (0mm,0mm) circle (3.5mm)
      ;
    \fill [fill=black!30]
      (-1.5mm,2.6mm) arc (120:300:3mm)
      ;
  \end{tikzpicture}}
  ;
\draw
  (-2,4) node {\begin{tikzpicture}
    \fill [fill=white]
      (0mm,0mm) circle (3.5mm)
      ;
    \fill [fill=black!30]
      (-1.25mm,2.165mm) arc (120:300:2.5mm)
      ;
  \end{tikzpicture}}
  ;
\draw
  (-5,1) node {\begin{tikzpicture}
    \fill [fill=white]
      (0mm,0mm) circle (3.5mm)
      ;
    \fill [fill=black!30]
      (-1.25mm,2.165mm) arc (120:300:2.5mm)
      ;
  \end{tikzpicture}}
  ;
\draw
  (-5,-1) node {\begin{tikzpicture}
    \fill [fill=white]
      (0mm,0mm) circle (3.5mm)
      ;
    \fill [fill=black!30]
      (-1.25mm,2.165mm) arc (120:300:2.5mm)
      ;
  \end{tikzpicture}}
  ;
\draw
  (-2,-2) node {\begin{tikzpicture}
    \fill [fill=white]
      (0mm,0mm) circle (3.5mm)
      ;
    \fill [fill=black!30]
      (-1.25mm,2.165mm) arc (120:300:2.5mm)
      ;
  \end{tikzpicture}}
  ;
\draw
  (4,-2) node {\begin{tikzpicture}
    \fill [fill=white]
      (0mm,0mm) circle (3.5mm)
      ;
    \fill [fill=black!30]
      (-1.25mm,2.165mm) arc (120:300:2.5mm)
      ;
  \end{tikzpicture}}
  ;

\draw
  (-8,-2) -- (-4,-2) -- (-6,0) -- cycle
  (-5,-1) -- (-1,-1) -- (-3,1) -- cycle
  (-5,1) -- (-1,1) -- (-3,3) -- cycle
  (-2,-2) -- (2,-2) -- (0,0) -- cycle
  (-2,0) -- (2,0) -- (0,2) -- cycle
  (-2,2) -- (2,2) -- (0,4) -- cycle
  (-2,4) -- (2,4) -- (0,6) -- cycle
  (1,-1) -- (5,-1) -- (3,1) -- cycle
  (1,1) -- (5,1) -- (3,3) -- cycle
  (4,-2) -- (8,-2) -- (6,0) -- cycle

  (-4.3,-2.3) -- (-3,-1) -- (-1.7,-2.3)
  (1.7,-2.3) -- (3,-1) -- (4.3,-2.3)
  (6.6,0) -- (4,0) -- (5.3,1.3)
  (3.6,3) -- (1,3) -- (2.3,4.3)
  (-6.6,0) -- (-4,0) -- (-5.3,1.3)
  (-3.6,3) -- (-1,3) -- (-2.3,4.3)

  (-6,-1.5) node {$T_{-2,-2}$}
  (-3,-.5) node {$T_{-1,-1}$}
  (-3,1.5) node {$T_{-1,1}$}
  (0,-1.5) node {$T_{0,-2}$}
  (0,.5) node {$T_{0,0}$}
  (0,2.5) node {$T_{0,2}$}
  (0,4.5) node {$T_{0,4}$}
  (3,-.5) node {$T_{1,-1}$}
  (3,1.5) node {$T_{1,1}$}
  (6,-1.5) node {$T_{2,-2}$}
  
  (-4,-1.17) node {\scriptsize $L_{-1,-1}$}
  (-2,-1.17) node {\scriptsize $R_{-1,-1}$}
  (-4,.83) node {\scriptsize $L_{-1,1}$}
  (-2,.83) node {\scriptsize $R_{-1,1}$}
  (-1,-.17) node {\scriptsize $L_{0,0}$}
  (1,-.17) node {\scriptsize $R_{0,0}$}
  (-1,1.83) node {\scriptsize $L_{0,2}$}
  (1,1.83) node {\scriptsize $R_{0,2}$}
  (-1,3.83) node {\scriptsize $L_{0,4}$}
  (1,3.83) node {\scriptsize $R_{0,4}$}
  (4,-1.17) node {\scriptsize $R_{1,-1}$}
  (2,-1.17) node {\scriptsize $L_{1,-1}$}
  (4,.83) node {\scriptsize $R_{1,1}$}
  (2,.83) node {\scriptsize $L_{1,1}$}
  (5,-.17) node {\scriptsize $L_{2,0}$}
  (-5,-.17) node {\scriptsize $R_{-2,0}$}
  (2,2.83) node {\scriptsize $L_{1,3}$}
  (-2,2.83) node {\scriptsize $R_{-1,3}$}
  ;
 
\draw (-2.22,-.07) node{\tikz{ \draw (0,0) circle (1.2mm) (0,0) node {\scriptsize 0};}};
\draw (2.22,-.07) node{\tikz{ \draw (0,0) circle (1.2mm) (0,0) node {\scriptsize 0};}};
\draw (0,1.4) node{\tikz{ \draw (0,0) circle (1.2mm) (0,0) node {\scriptsize 0};}};
\draw (1,-.65) node{\tikz{ \draw (0,0) circle (1.2mm) (0,0) node {\scriptsize 0};}};
\draw (-1,-.65) node{\tikz{ \draw (0,0) circle (1.2mm) (0,0) node {\scriptsize 0};}};
 
\draw (0,2.15) node{\tikz{ \draw (0,0) circle (1.2mm) (0,0) node {\scriptsize 1};}};
\draw (-1,1.35) node{\tikz{ \draw (0,0) circle (1.2mm) (0,0) node {\scriptsize 1};}};
\draw (1,1.35) node{\tikz{ \draw (0,0) circle (1.2mm) (0,0) node {\scriptsize 1};}};
\draw (-2,.35) node{\tikz{ \draw (0,0) circle (1.2mm) (0,0) node {\scriptsize 1};}};
\draw (2,.35) node{\tikz{ \draw (0,0) circle (1.2mm) (0,0) node {\scriptsize 1};}};
 
\draw (-3,.4) node{\tikz{ \draw (0,0) circle (1.2mm) (0,0) node {\scriptsize 2};}};
\draw (3,.4) node{\tikz{ \draw (0,0) circle (1.2mm) (0,0) node {\scriptsize 2};}};
\draw (0,3.4) node{\tikz{ \draw (0,0) circle (1.2mm) (0,0) node {\scriptsize 2};}};
\draw (-2.22,1.93) node{\tikz{ \draw (0,0) circle (1.2mm) (0,0) node {\scriptsize 2};}};
\draw (2.22,1.93) node{\tikz{ \draw (0,0) circle (1.2mm) (0,0) node {\scriptsize 2};}};
\draw (-3,1.15) node{\tikz{ \draw (0,0) circle (1.2mm) (0,0) node {\scriptsize 2};}};
\draw (3,1.15) node{\tikz{ \draw (0,0) circle (1.2mm) (0,0) node {\scriptsize 2};}};
\draw (-.78,-1.07) node{\tikz{ \draw (0,0) circle (1.2mm) (0,0) node {\scriptsize 2};}};
\draw (.78,-1.07) node{\tikz{ \draw (0,0) circle (1.2mm) (0,0) node {\scriptsize 2};}};
 
\end{tikzpicture}
\end{center}
\caption{Construction of $\mathcal{T}_0$.}
\label{fig_6}
\end{figure}
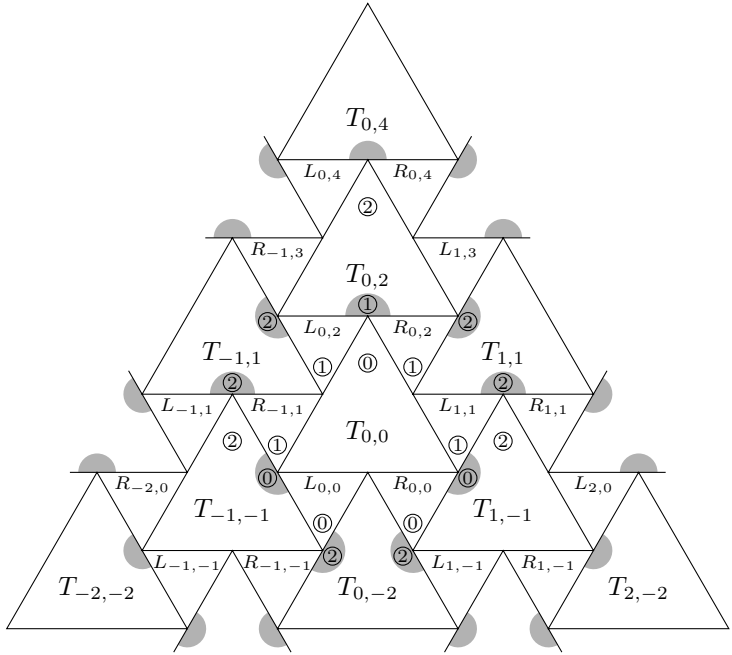

\emph{Step 0. } We pick a maximal segment $S \subseteq \SKEL(\mathcal{T})$. By part (B), we can assume that the situation close to $S$ is as in Figure~\ref{fig_5}. We define $T_{0,0}=T_{N,1}$, $L_{0,0}=T_{S,1}$ and $R_{0,0}=T_{S,2}$.

\emph{Step 1. } The north east side $S_1$ of $T_{0,0}$ is parallel to $v_{NW}$ and is bounded at its south east endpoint by a segment in direction $v_{NE}$. By (B), this implies that $S_1$ is a maximal line segment of $\SKEL(\mathcal{T})$, $S_1$ is bounded at its north west endpoint by a segment in direction $v_E$, and $S_1$ is the union of two sides of triangles $L_{1,1},R_{0,2} \in \mathcal{T}$ opposite to $T_{0,0}$ (where $R_{0,2}$ is in the north west of $L_{1,1}$). 
Similarly, the north west side of $T_{0,0}$ is the union of two sides of triangles $R_{-1,1},L_{0,2} \in \mathcal{T}$ opposite to $T_{0,0}$.

\emph{Step 2. } Now consider the segment $S_2$ formed by the collinear sides of $L_{0,2}$ and $R_{0,2}$. By (B), $S_2$ is a maximal segment in $\SKEL(\mathcal{T})$, $S_2$ is bounded at its west endpoint by a segment in direction $v_{NW}$ and at its east endpoint by a segment in direction $v_{NE}$, and $S_2$ is a side of some $T_{0,2} \in \mathcal{T}$ opposite to $L_{0,2}$ and $R_{0,2}$. Similar considerations apply to segments formed by collinear sides of $R_{-1,1}$ and $L_{0,0}$ as well as of $R_{0,0}$ and $L_{1,1}$. In particular, we obtain tiles $T_{-1,-1}$ opposite to $R_{-1,1}$ and $L_{0,0}$ and $T_{1,-1}$ opposite to $R_{0,0}$ and $L_{1,1}$.

\emph{Step 3. } Now we know that one side of each of $T_{0,2}$, $T_{-1,-1}$ and $T_{1,-1}$ is in the same topological situation as was the initial segment $S$ with respect to $T_{0,0}$ after Step 0. We repeat Steps 1 and 2 with these three maximal segments and arrive at the topological structure displayed in Figure~\ref{fig_6}.

Iterating this procedure, we get the desired tiling $\mathcal{T}_0 \subseteq \mathcal{T}$. 
\end{proof}


\subsection{The remaining topological type and a random walk}

We associate a directed graph $\Gamma=\left(\ZE,E\right)$ with vertex set $\ZE$ to the triangles $T_{i,j} \in \mathcal{T}_0$ constructed above. Two vertices $(i,j),(k,l) \in \ZE$ are connected by an edge, $[(i,j),(k,l)] \in E$, if and only if some vertex of $T_{i,j}$ belongs to $\BD(T_{k,l})$. That is,
$$
[(i,j),(k,l)] \in E \quad\Longleftrightarrow\quad (k,l) \in \{(i-1,j-1),(i+1,j-1),(i,j+2)\}
$$
(see Figure~\ref{fig_7}).
\begin{figure}

\begin{tikzpicture}[xscale=2.2,yscale=.6]

\draw
  (-2,-2) node {$(-2,-2)$}
  (-2,0) node {$(-2,0)$}
  (-2,2) node {$(-2,2)$}
  (-1,1) node {$(-1,1)$}
  (-1,-1) node {$(-1,-1)$}
  (-1,-3) node {$(-1,-3)$}
  (0,-2) node {$(0,-2)$}
  (0,0) node {$(0,0)$}
  (0,2) node {$(0,2)$}
  (1,1) node {$(1,1)$}
  (1,-1) node {$(1,-1)$}
  (1,-3) node {$(1,-3)$}
  (2,-2) node {$(2,-2)$}
  (2,0) node {$(2,0)$}
  (2,2) node {$(2,2)$}
  ;

\draw[<-] (-2,1.5) -- (-2,.5); 
\draw[<-] (-2,-.5) -- (-2,-1.5); 
\draw[<-] (-1,.5) -- (-1,-.5); 
\draw[<-] (-1,-1.5) -- (-1,-2.5); 
\draw[<-] (0,1.5) -- (0,.5); 
\draw[<-] (0,-.5) -- (0,-1.5); 
\draw[<-] (1,.5) -- (1,-.5); 
\draw[<-] (1,-1.5) -- (1,-2.5); 
\draw[<-] (2,1.5) -- (2,.5); 
\draw[<-] (2,-.5) -- (2,-1.5); 

\draw[->] (-1.7,1.7) -- (-1.3,1.3); 
\draw[->] (-1.3,.7) -- (-1.7,.3); 
\draw[->] (-1.7,-.3) -- (-1.3,-.7); 
\draw[->] (-1.3,-1.3) -- (-1.7,-1.7); 
\draw[->] (-1.7,-2.3) -- (-1.3,-2.7); 

\draw[->] (-.3,1.7) -- (-.7,1.3); 
\draw[->] (-.7,.7) -- (-.3,.3); 
\draw[->] (-.3,-.3) -- (-.7,-.7); 
\draw[->] (-.7,-1.3) -- (-.3,-1.7); 
\draw[->] (-.3,-2.3) -- (-.7,-2.7); 

\draw[->] (.3,1.7) -- (.7,1.3); 
\draw[->] (.7,.7) -- (.3,.3); 
\draw[->] (.3,-.3) -- (.7,-.7); 
\draw[->] (.7,-1.3) -- (.3,-1.7); 
\draw[->] (.3,-2.3) -- (.7,-2.7); 

\draw[->] (1.7,1.7) -- (1.3,1.3); 
\draw[->] (1.3,.7) -- (1.7,.3); 
\draw[->] (1.7,-.3) -- (1.3,-.7); 
\draw[->] (1.3,-1.3) -- (1.7,-1.7); 
\draw[->] (1.7,-2.3) -- (1.3,-2.7); 

\end{tikzpicture}

\caption{The directed graph $\Gamma=\left(\ZE,E\right)$.}
\label{fig_7}
\end{figure}
We consider the random walk on $\Gamma$ with transition probability $p((i,j),(k,l))=\frac{1}{3}$ for $[(i,j),(k,l)] \in E$ (and in turn $p((i,j),(k,l))=0$ for $[(i,j),(k,l)] \notin E$). This is described by the Markov chain $\left(\ZE,P\right)$ with state space $\ZE$ and transition matrix $P=(p((i,j),(k,l)))_{(i,j),(k,l) \in \ZE}$ (cf. \cite[pp. 3--11]{woess2009}).

Note that, for any two states $(i,j),(k,l) \in \ZE$, there is a path from $(i,j)$ to $(k,l)$ in $\Gamma$. Equivalently, the probability to reach $(k,l)$ within finitely many steps when starting at $(i,j)$ is positive. Such Markov chains are called \emph{irreducible} \cite[p. 28]{woess2009}.

The probability $p^{(n)}((i,j),(k,l))$ of reaching the state $(k,l)$ from the state $(i,j)$ in exactly $n$ steps, $n \in \{0,1,\ldots\}$, is the respective entry of the $n$th power $P^n$ of $P$ \cite[pp. 12--13, Lemma 1.21(a)]{woess2009}. That is,
$$
p^{(n)}((i,j),(k,l))= \sum_\pi p((i,j),(i_1,j_1))p((i_1,j_1),(i_2,j_2))\ldots p((i_{n-1},j_{n-1}),(k,l))
$$
with summation over all paths $\pi= [(i,j),(i_1,j_1),(i_2,j_2),\ldots,(i_{n-1},j_{n-1}),(k,l)]$ of length $n$ in $\Gamma$ (i.e.,
$[(i,j),(i_1,j_1)],[(i_1,j_1),(i_2,j_2)],\ldots,[(i_{n-1},j_{n-1}),(k,l)] \in E$).

In particular, the probability $p^{(n)}((i,j),(i,j))$ of returning from $(i,j)$ to $(i,j)$ in exactly $n$ steps is
$$
p^{(n)}((i,j),(i,j))= \sum_\pi \left(\frac{1}{3}\right)^n
$$
with summation over all cycles $\pi$ of length $n$ that start at $(i,j)$.

Recall that every edge in $E$ is of the form $[(k,l),(k,l)+(-1,-1)]$, $[(k,l),(k,l)+(1,-1)]$ or $[(k,l),(k,l)+(0,2)]$. Since the vectors $(-1,-1)$, $(1,-1)$ and $(0,2)$ are pairwise linearly independent, every cycle must have length $n=3m$ for some $m\in\{0,1,\dots\}$ and contain $m$ edges of each of these three types. Therefore the number of cycles of length $3m$ starting at $(i,j)$ is $\frac{(3m)!}{m!m!m!}$, which counts the possible consecutive orders of these three types of edges. Thus
$$
p^{(n)}((i,j),(i,j))=\begin{cases}
\frac{(3m)!}{3^{3m}m!m!m!} & \text{if } n=3m,\; m \in \{0,1,\ldots\},\\
0 & \text{otherwise.}
\end{cases}
$$

The expected number of visits of the state $(i,j)$, provided the random walk starts at $(i,j)$, is
$$
G((i,j),(i,j))= \sum_{n=0}^\infty p^{(n)}((i,j),(i,j))= \sum_{m=0}^\infty \frac{(3m)!}{3^{3m}(m!)^3}
$$
\cite[p. 15]{woess2009}. Using Stirling's estimates $\sqrt{2\pi}m^{m+\frac{1}{2}}e^{-m} \le m! \le e m^{m+\frac{1}{2}}e^{-m}$ for $m \in \{1,2,\ldots\}$, we obtain
$$
G((i,j),(i,j)) \ge \sum_{m=1}^\infty \frac{\sqrt{2\pi} (3m)^{3m+\frac{1}{2}}e^{-3m}}{3^{3m} \left(em^{m+\frac{1}{2}}e^{-m}\right)^3} = \frac{\sqrt{6\pi}}{e^3}\sum_{m=1}^\infty \frac{1}{m}= \infty
$$
for all $(i,j) \in \ZE$.

A state $(i,j)$ is called \emph{recurrent} if $G((i,j),(i,j))=\infty$ \cite[Theorem 3.4(a) and Definition 3.1]{woess2009}, and a Markov chain is called recurrent if all its states are recurrent. An irreducible Markov chain is recurrent if and only if one of its states is recurrent \cite[pp. 45--46, 48]{woess2009}. 

\begin{lemma}
The Markov chain $\left(\ZE,P\right)$ is recurrent.
\end{lemma}

A function $f:\ZE \to \mathbb{R}$ is called \emph{harmonic} with respect to $\left(\ZE,P\right)$ if
it satisfies the mean value property
$f(i,j)=\sum_{(k,l) \in \ZE} p((i,j),(k,l))f(k,l)$ for all $(i,j) \in \ZE$ \cite[Definition 6.13]{woess2009}; i.e.,
\begin{equation}\label{eq_harmonic}
f(i,j)= \frac{1}{3}\big(f(i-1,j-1)+f(i+1,j-1)+f(i,j+2)\big).
\end{equation}
By \cite[Theorem 6.21]{woess2009}, every non-negative harmonic function of a recurrent irreducible Markov chain is constant.

\begin{corollary}\label{cor_1}
Every non-negative harmonic function of $\left(\ZE,P\right)$ is constant.
\end{corollary}

We come back to tilings by equilateral triangles.

\begin{lemma}\label{lem_4}
If a tiling $\mathcal{T}_0$ by equilateral triangles has the same topology as the one from Figure~\ref{fig_1}, then $\mathcal{T}_0$ is the image of a tiling of the family depicted in Figure~\ref{fig_1} under a similarity transformation of the plane.
\end{lemma}

\begin{proof}
Adopting notation from the above construction of $\mathcal{T}_0$ (see Figure~\ref{fig_6}), we suppose that
$\mathcal{T}_0$ consists of the triangles $T_{i,j},L_{i,j},R_{i,j}$ with $(i,j) \in \ZE$ and that the three sides of $T_{i,j}$ split into the collinear sides of $L_{i,j}$ and $R_{i,j}$, of $L_{i+1,j+1}$ and $R_{i,j+2}$ and of $L_{i,j+2}$ and $R_{i-1,j+1}$, respectively. In particular,
\begin{align}
\DIAM(T_{i,j}) &=\DIAM(L_{i,j})+\DIAM(R_{i,j}) \label{eq_sum1}\\ 
&=\DIAM(L_{i+1,j+1})+\DIAM(R_{i,j+2}) \label{eq_sum2}\\
&=\DIAM(L_{i,j+2})+\DIAM(R_{i-1,j+1}). \label{eq_sum3}
\end{align}
Summing these three equations and then applying (\ref{eq_sum2}) to $T_{i-1,j-1}$, (\ref{eq_sum3}) to $T_{i+1,j-1}$ and (\ref{eq_sum1}) to $T_{i,j+2}$, we obtain
\begin{align*}
3\DIAM(T_{i,j}) &=\big(\DIAM(L_{i,j})+\DIAM(R_{i-1,j+1})\big)+\big(\DIAM(L_{i+1,j+1})+\DIAM(R_{i,j})\big)\\
&\phantom{=(}
+\big(\DIAM(L_{i,j+2})+\DIAM(R_{i,j+2})\big)\\
&=\DIAM(T_{i-1,j-1})+\DIAM(T_{i+1,j-1})+\DIAM(T_{i,j+2}).
\end{align*}
Hence the function $f:\ZE \to \mathbb{R}$, $f(i,j)=\DIAM(T_{i,j})$, satisfies the harmonicity condition (\ref{eq_harmonic}). As $f(i,j)=\DIAM(T_{i,j})>0$ for all $(i,j)\in\ZE$, Corollary~\ref{cor_1} shows that $f$ is constant. Since we have to characterize $\mathcal{T}_0$ only up to similarity, we can assume that
\begin{equation}\label{eq_T}
\DIAM(T_{i,j})=1 \quad\mbox{for all}\quad (i,j) \in \ZE.
\end{equation}

Now
\begin{align}
\DIAM(L_{i,j+2}) &\stackrel{\text{(\ref{eq_sum3})}}{=} \DIAM(T_{i,j})-\DIAM(R_{i-1,j+1}) \nonumber\\
&\stackrel{\text{(\ref{eq_sum2})}}{=} \DIAM(T_{i,j})-\big(\DIAM(T_{i-1,j-1})-\DIAM(L_{i,j})\big) \nonumber\\
&\stackrel{\text{(\ref{eq_T})}}{=} \DIAM(L_{i,j})\label{eq_L1} 
\end{align}
and
\begin{align}
\DIAM(L_{i+1,j+1}) &\stackrel{\text{(\ref{eq_sum3})}}{=} \DIAM(T_{i+1,j-1})-\DIAM(R_{i,j}) \nonumber\\
&\stackrel{\text{(\ref{eq_sum1})}}{=} \DIAM(T_{i+1,j-1})-\big(\DIAM(T_{i,j})-\DIAM(L_{i,j})\big) \nonumber\\
&\stackrel{\text{(\ref{eq_T})}}{=} \DIAM(L_{i,j})\label{eq_L2}.
\end{align}
W.l.o.g., $\DIAM(L_{0,0}) \le \DIAM(R_{0,0})$. Then the parameter $\alpha=\DIAM(L_{0,0})$ satisfies
\begin{equation}\label{eq_alpha}
0 < \alpha \le \frac{1}{2},
\end{equation}
because $\DIAM(L_{0,0})+\DIAM(R_{0,0})\stackrel{\text{(\ref{eq_sum1})}}{=}\DIAM(T_{0,0})\stackrel{\text{(\ref{eq_T})}}{=}1$,
and the periodicities (\ref{eq_L1}) and (\ref{eq_L2}) yield
\begin{equation}\label{eq_L}
\DIAM(L_{i,j})=\alpha \quad\mbox{for all}\quad (i,j) \in \ZE.
\end{equation}
Finally, (\ref{eq_sum1}), (\ref{eq_T}) and (\ref{eq_L}) imply
\begin{equation}\label{eq_R}
\DIAM(R_{i,j})=1-\alpha \quad\mbox{for all}\quad (i,j) \in \ZE.
\end{equation}

Claims (\ref{eq_T}), (\ref{eq_L}), (\ref{eq_R}) and (\ref{eq_alpha}) prove Lemma~\ref{lem_4}.
\end{proof}

\begin{proof}[Proof of Proposition~\ref{prop_2}]
Lemmas~\ref{lem_3} and \ref{lem_4} show that $C=\mathbb{R}^2$ and that $\mathcal{T}$ contains a tiling $\mathcal{T}_0$ that is a similar image of one illustrated in Figure~\ref{fig_1}. We obtain
$\mathcal{T}=\mathcal{T}_0$, because $\bigcup\mathcal{T}_0=\mathbb{R}^2=\bigcup\mathcal{T}$.
\end{proof}


\subsection{Conclusion}

\begin{proof}[Proof of Theorem~\ref{thm_main}]
Lemmas~\ref{lem_1} and \ref{lem_2} discuss tilings that are not locally finite and locally finite tilings possessing $E$-configurations. Propositions~\ref{prop_1} and \ref{prop_2} exclude locally finite tilings without $E$-configurations for bounded and for unbounded convex sets, respectively, by showing that they cannot be perfect.
\end{proof}

\begin{proof}[Proof of Theorem~\ref{thm_shared-side}]
By Lemma \ref{lem_1}, the tiling $\mathcal{T}$ is necessarily locally finite. Lemma~\ref{lem_2} settles the case of tilings with $E$-configurations. Propositions~\ref{prop_1} and \ref{prop_2} discuss tilings without $E$-configurations.
\end{proof}

Finally, let us point out that Proposition~\ref{prop_2} and the proof of Lemma~\ref{lem_2} give rise to the following necessary condition in the context of Problem~\ref{prob}.

\begin{corollary}\label{cor_condition}
Every locally finite perfect tiling of an unbounded convex subset of the plane $\mathbb{R}^2$ by equilateral triangles possesses a sequence of $E$-configurations of strictly decreasing lengths.
\end{corollary}


\bibliographystyle{plain}


\end{document}